\documentclass[paper=letter,DIV=8]{scrartcl}
\usepackage{pgffor,graphicx}
\usepackage{amsmath,amsfonts,amssymb,amsthm}
\usepackage{braket}
\usepackage[inline]{asymptote}

\usepackage[english]{babel}
\usepackage[utf8]{inputenc}
\newcommand{\eps}{\varepsilon}
\usepackage{mathtools}
\makeatletter 
\InputIfFileExists{\jobname.pre2}{}{}
\makeatother

\usepackage{csquotes}
\usepackage[unicode]{hyperref}
\usepackage[hypcap=true]{caption}
\usepackage[hypcap=true]{subcaption}

\begin{asydef}
    import lib;
    import graph_settings;
    defaultpen (fontsize (10));
    unitsize (15mm);
\end{asydef}

\makeatletter
\foreach \letter in {A,B,C,D,E,F,G,H,I,J,K,L,M,N,O,P,Q,R,S,T,U,V,W,X,Y,Z}{
  \expandafter\expandafter\expandafter\csxdef\expandafter{bb\letter}{\noexpand\mathbb{\letter}}
  \expandafter\expandafter\expandafter\csxdef\expandafter{mb\letter}{\noexpand\mathbf{\letter}}
  \expandafter\expandafter\expandafter\csxdef\expandafter{mc\letter}{\noexpand\mathcal{\letter}}
}
\makeatother

\newcommand{\WG}{\mathbf{WG}}
\newcommand{\LEG}{\mathbf{LEG}}
\newcommand{\SD}{\mathbf{SD}}

\usepackage{datetime,scrlayer-scrpage}
\ohead{\headmark}
\ihead{Families with many invariants}
\automark[subsection]{section}
\usepackage[giveninits=true,maxnames=10,backend=biber,sortcites]{biblatex}
\usepackage[inline]{enumitem}

\newtheorem{theorem}{Theorem} 
\newtheorem{lemma}{Lemma} 
\newtheorem{conjecture}{Conjecture} 
\newtheorem{corollary}{Corollary} 
\newtheorem{proposition}{Proposition} 
 
\theoremstyle{definition}
\newtheorem{definition}{Definition} 
\newtheorem*{convention}{Convention}
\theoremstyle{remark}
\newtheorem{remark}{Remark} 

\DeclareMathOperator{\Vect}{Vect}
\DeclareMathOperator{\codim}{codim}

\AtBeginDocument{

}

\author{Nataliya Goncharuk\thanks{University of Toronto at Mississauga, 3359 Mississauga Road, Mississauga, ON L5L 1C6}\thanks{The work of both authors is partially supported by RFBR grant No. 20-01-00420.} \and Yury G. Kudryashov\footnotemark[1]}
\title{Families of vector fields with many numerical invariants}
\hypersetup{
 pdfauthor={Nataliya Goncharuk, Yury G. Kudryashov},
 pdftitle={Families of vector fields with many numerical invariants},
 pdfkeywords={bifurcations, vector fields},
 pdfsubject={},
 pdflang={English}}

\begin{document}

\maketitle

\begin{abstract}
    We study bifurcations in finite-parameter families of vector fields on \(S^2\).
    Recently, Yu.\,Ilyashenko, Yu.\,Kudryashov, and I.\,Schurov provided examples of (locally generic) structurally unstable \(3\)-parameter families of vector fields: topological classification of these families admits at least one numerical invariant.
    They also provided examples of \((2D+1)\)-parameter families such that the topological classification of these families has at least \(D\) numerical invariants and used those examples to construct families with functional invariants of topological classification.

    In this paper, we construct locally generic \(4\)-parameter families with any prescribed number of numerical invariants and use them to construct \(5\)-parameter families with functional invariants of the form \((\bbR, a)\to(\bbR^{D}, b)\). We also describe a locally generic class of \(3\)-parameter families with infinitely many numerical invariants of topological classification. As far as we know, this example does not lead to families with functional invariants.
\end{abstract}

\section{Introduction}%
\label{sec:intro}

It is well-known that a generic vector field on $S^2$ is structurally stable, see~\cite{AP37,Ba52,P59,P62,P63,ALGM66:en,ALGM67:en}.
One of the goals of bifurcation theory is to study and classify generic bifurcations of \emph{degenerate} planar vector fields in generic finite-parameter families.

V.\,Arnold\footnote{The conjecture appears in~\cite[Sec. I.3.2.8]{AAIS94} with the following remark: \enquote{Certainly proofs or counterexamples to the above conjectures are necessary for investigating nonlocal bifurcations in generic \(l\)-parameter families}. According to Yu.\,Ilyashenko, this section was written entirely by V.\,Arnold.} conjectured that a generic $k$-parameter family of vector fields on $S^2$ is structurally stable: generic close families experience equivalent bifurcations (see \autoref{sec:equivs} for different formalizations of the notion of equivalent bifurcations).
This conjecture turned out to be wrong.
In~\cite{IKS-th1}, Yu.\,Ilyashenko, Yu.\,Kudryashov, and I.\,Schurov show that locally generic 3-parameter unfoldings of the polycycle \enquote{tears of the heart} (see \autoref{fig:TH}) are structurally unstable.
In~\cite{GK-phase}, we show that the classification of such unfoldings may have arbitrarily many numerical invariants.

\begin{figure}
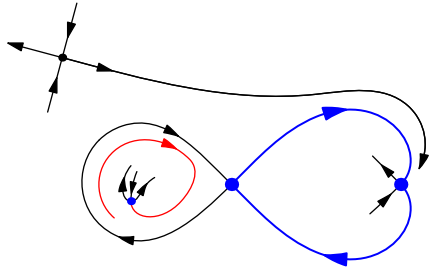

    \centering
    \asyinclude{TH}
    \caption{A vector field with a polycycle \enquote{tears of the heart}}%
    \label{fig:TH}
\end{figure}

In~\cite{GKS-glasses}, N.\,Solodovnikov and we prove that 3-parameter unfoldings of vector fields with \enquote{ears} and \enquote{glasses} separatrix graphs (see \autoref{fig:ears-glasses}) are also structurally unstable, and the classification of these unfoldings has numerical invariants.

In this paper, we construct several degenerate vector fields such that \enquote{ears} and \enquote{glasses} separatrix graphs are generated in their unfoldings.
This enables us to construct the following.
\begin{description}
  \item[\autoref{thm:pc:seq}:] Generic 3-parameter families with infinitely many invariants of classification.
  \item[\autoref{cor:pc:seq}:] The class of degenerate (codimension \(3\)) vector fields such that each vector field of this class has infinitely many different $3$-parameter unfoldings.
    The classification of these unfoldings has infinitely many invariants.
  \item[\autoref{thm:LEG-num}:] Generic 4-parameter families with arbitrarily many invariants of classification.
  \item[\autoref{thm:LEG-func}:] Generic 5-parameter families with functional invariants of classification.
\end{description}

It looks like \autoref{thm:pc:seq} is stronger than \autoref{thm:LEG-num}.
However invariants in \autoref{thm:pc:seq} and \autoref{thm:LEG-num} have a different nature.
In \autoref{thm:pc:seq}, invariants depend on the unfolding of a degenerate vector field.
In particular, different generic unfoldings of the same vector field are not equivalent, see \autoref{cor:pc:seq}.
However when we add more parameters, these invariants disappear.

Invariants in \autoref{thm:LEG-num} only depend on the unperturbed vector field and survive when we add more parameters.
We use this to construct functional invariants in 5-parameter families for \autoref{thm:LEG-func}, using a method developed in~\cite{IKS-th1}, see also \autoref{cor:semi-stable-functional}.
\autoref{thm:LEG-func} improves~\cite[Theorem 2]{IKS-th1}, where functional invariants were discovered for 6-parameter unfoldings of the \enquote{tear-luna-heart} polycycle.
Recently, A.\,Dukov and Ilyashenko found another construction that leads to 5-parameter families with functional invariants~\cite{Dukov-nonlocal,DuYuI-nonlocal}.

\begin{figure}
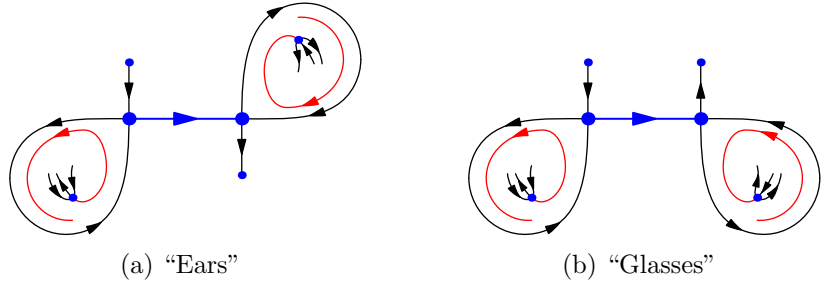

    \centering
    \subcaptionbox{\label{fig:ears}\enquote{Ears}}{\asyinclude{ears}}\hfil
    \subcaptionbox{\label{fig:glasses}\enquote{Glasses}}{\asyinclude{glasses}}
    \caption{\label{fig:ears-glasses}
      Vector fields with \enquote{ears} and \enquote{glasses}
      }
\end{figure}

\section{Preliminaries}%
\label{sec:prelims}

\subsection{Families and equivalences}%
\label{sec:equivs}
By \(\Vect\) we denote the Banach space of \(C^3\)-smooth vector fields on the two-sphere.
\begin{definition}
    Given a~non-empty open subset \(\mcB\subset\bbR^k\), a map \(V\colon \mcB\to\Vect\), \(V=\set{v_{\alpha}}_{\alpha\in\mcB}\) is called a \emph{\(k\)-parameter family of vector fields}.
    A \emph{local family} is a germ of a map \(V \colon (\bbR^k, 0)\to(\Vect, v_0)\).
    Denote by \(\mcV_{k}\) the Banach space of local families \(V=\set{v_{\alpha}}_{\alpha\in(\bbR^{k},0)}\) that are \(C^{3}\)-smooth in \((\alpha, x)\).
\end{definition}

\begin{definition}%
    \label{def:vf-eq}
    Vector fields  \(v,w\in \Vect\) are called \emph{orbitally topologically equivalent} if there exists an orientation preserving homeomorphism \(H\colon S^2\to S^2\) that takes phase curves of \(v\) to phase curves of \(w\) and preserves time orientation.
\end{definition}

\begin{definition}%
    \label{def:weak-eq}
    Local families \(V=\set{v_{\alpha }}_{\alpha \in (\bbR^{n},0)}\) and \(\tilde V=\set{\tilde v_{\tilde \alpha }}_{\tilde \alpha \in (\bbR^{n},0)}\) are called \emph{weakly topologically equivalent} if there exists a germ of a homeomorphism \(h\colon (\bbR^{n},0)\to (\bbR^{n},0)\) of the parameter spaces such that for sufficiently small $\alpha$, vector fields $v_\alpha$ and $\tilde v_{h(\alpha)}$ are orbitally topologically equivalent.
\end{definition}
The notion of weak equivalence does not completely agree with our intuitive understanding of \enquote{equivalent} bifurcations. When we study bifurcations, we use descriptions like \enquote{this degenerate singular point of $v_0$ splits into two}, \enquote{these limit cycles of $v_{\alpha}$ land on that polycycle of $v_0$} etc.  Such  descriptions do not survive under weak equivalence, because a continuous family of  singular points (or limit cycles) of $v_{\alpha}$ can be mapped to a discontinuous family of singular points (limit cycles) of $\tilde v_{h(\alpha)}$.

This motivated the notion of \emph{moderate} equivalence proposed by Yu.\,Ilyashenko in~\cite{IKS-th1}.
For moderate equivalence, the conjugating homeomorphism $H_\alpha$ is required to be continuous with respect to $\alpha$ on a certain \enquote{interesting} part of the phase portrait.
Moderate equivalence respects the descriptions of bifurcations mentioned above.
Moderate equivalence and other similar equivalences were used in~\cite{IKS-th1, GI-LBS,GK-phase,GKS-glasses,GIS-semistable}.
For a comparison of different equivalence relations on the space of families of vector fields see~\cite{GI-equivs}.

The formal statement of the above mentioned Arnold's conjecture is the following:
\begin{conjecture}[Arnold, {\cite[Sec. I.3.2.8]{AAIS94}}]%
    \label{conj-Arnold}
    Generic finite-parameter families of vector fields on the two-sphere are structurally stable: close families are weakly topologically equivalent.
\end{conjecture}

Arnold's conjecture was disproved in~\cite{IKS-th1} for \emph{moderate} equivalence instead of weak equivalence of families.

In this paper we will use the definition of \emph{weak topological equivalence with $Sep$-tracking}, also used in~\cite{GKS-glasses}.
The idea is to \enquote{track} a certain set of marked saddles and separatrices of vector fields, and to consider conjugacies that take marked saddles and  separatrices of $v_\eps$ to marked saddles and separatrices of $\tilde v_{h(\eps)}$.
The formal definition appears in the next section.

\subsection{Vector fields with marked saddles and separatrices}%
\label{sub:marked}

Consider two weakly equivalent local families of vector fields \(V\), \(\tilde V\).
Suppose that we want to prove that this equivalence implies equality of some numerical quantities.
Quite often the proof relies on the fact that the orbital topological equivalence \(H_{\alpha }\) of \(v_{\alpha}\) and \(\tilde v_{h(\alpha)}\) sends some saddle points and separatrices of~\(v_{\alpha }\) to the \enquote{corresponding} saddle points and separatrices of~\(\tilde v_{h(\alpha )}\).

Formally, let \(\widehat\Vect\) be the Banach manifold of vector fields \(v\in\Vect\) with an ordered tuple of distinct hyperbolic saddles \(S_{1}, \dotsc , S_{N}\) and their local separatrices \((\gamma _{i,j}, S_{i})\), \(i=1,\dotsc ,N\), \(j=1,\dotsc ,4\).
We call these saddles and separatrices \emph{marked}.
Let \(\widehat\mcV_{k}\) be the set of \(k\)-parameter local families of vector fields with marked hyperbolic saddles and separatrices;
formally, \(V\in\mcV_{k}\) is a germ of a continuous map \(V\colon (\bbR^{k}, 0)\to (\widehat\Vect, v_{0})\) such that \(v_{\alpha}(x)\) is \(C^{3}\)-smooth in \((\alpha, x)\), and the marked saddles and separatrices are continuous in \(\alpha\).
Given a family \(V\in\mcV_{k}\) and a choice of marked saddles and separatrices of \(v_{0}\), there is a unique way to lift this family to a family \(\hat V\in\widehat\mcV_{k}\) that agrees with this choice.

The following definition is a natural modification of Definitions~\ref{def:vf-eq} and~\ref{def:weak-eq} for vector fields with marked saddles and separatrices.
For bifurcations studied in this paper, the definition below is the best approximation to our intuitive concept of \enquote{same} bifurcations.
Unlike the weak equivalence, it respects the description of bifurcations in terms of the behaviour of saddle separatrices of perturbed vector fields.

\begin{definition}%
    \label{def:tracking}
    Two vector fields with marked saddles and separatrices \(v, \tilde v\in \widehat\Vect\) are called \emph{orbitally topologically equivalent with \(Sep\)-tracking}, if there exists an orientation preserving homeomorphism \(H\colon S^{2}\to S^{2}\) that sends trajectories of \(v\) to trajectories of \(\tilde v\), each marked saddle \(S_{i}\) of \(v\) to the corresponding saddle \(\tilde S_{i}\) of \(\tilde v\), and each marked separatrix \(\gamma_{i,j}\) of \(v\) to the corresponding marked separatrix of \(\tilde v\).

    Two local families \(V\colon (\bbR^{k}, 0)\to(\widehat\Vect, v_{0})\), \(\tilde V\colon (\bbR^{k}, 0)\to(\widehat\Vect, \tilde v_{0})\) of vector fields with marked saddles and separatrices are called \emph{weakly topologically equivalent with \(Sep\)-tracking}, if there exists a germ of a homeomorphism \(h\colon(\bbR^{k},0)\to(\bbR^{k},0)\) such that \(v_{\alpha}\) is orbitally topologically equivalent with \(Sep\)-tracking to \(\tilde v_{h(\alpha)}\).
\end{definition}

We can obtain this tracking property in two different ways.
One approach is to add this as an explicit requirement on the weak equivalence in the statement of a theorem.

Another approach would be to add extra elements (\enquote{tags}) to the phase portraits of~\(v_{\alpha }\) so that any weak equivalence between \(V\) and~\({\tilde V}\) will automatically track the marked saddles and separatrices.
For example, we can add several nests with different amounts of hyperbolic limit cycles.
Then the weak equivalence must map the nests of~\(V\) to the corresponding nests of \({\tilde V}\), and we can identify separatrices by the nests they wind onto.

With the latter approach, we can use the classical definition of weak equivalence instead of \autoref{def:tracking}, but we need to add extra \enquote{tags} to the phase portrait, and these \enquote{tags} do not actually bifurcate.

\begin{convention}
    In this paper we will only formulate and prove theorems about weak topological equivalence with \(Sep\)-tracking.
    In particular, whenever we say that two families \(V, \tilde V\colon(\bbR^{k}, 0)\to\widehat\Vect\) are equivalent, we mean that they are weakly topologically equivalent with \(Sep\)-tracking.
\end{convention}

\subsection{Invariant functions and numerical invariants}
Let \(\mbM\subset\widehat\Vect\) be a Banach submanifold, \(\codim \mbM<\infty\);
let \(k\) be a natural number, \(k\ge\codim \mbM\).
Denote~by~\(\mbM^{\pitchfork,k}\subset\widehat\mcV_{k}\) the set of local families~\(V\) such that \(v_0\in\mbM\) and \(V\) is transverse to~\(\mbM\) at~\(v_0\).
All numerical invariants of local families constructed in~\cite{IKS-th1,GK-phase,GKS-glasses} follow the same pattern:
they have the form \(V\mapsto\varphi(v_{0})\), where \(\varphi\colon \mbM\to\bbR\) is an \emph{invariant function} in the following sense.
\begin{definition}
    [{cf.~\cite[Definition 16]{IKS-th1}}]%
    \label{def:inv-func}
    A~function \(\varphi\colon \mbM\to \bbR\) defined on a Banach submanifold of~\(\widehat\Vect\) is called \emph{invariant}, if for any two equivalent local families \(V, \tilde V\in\mbM^{\pitchfork,\codim \mbM}\) we have \(\varphi(v_{0})=\varphi(\tilde v_{0})\).

    A function \(\varphi\colon \mbM\to \bbR\) is called \emph{robustly invariant}, if the same equality holds for any two equivalent families \(V, \tilde V\in\mbM^{\pitchfork,k}\) with \(k\ge\codim \mbM\).
\end{definition}

The inequality \(k\ge\codim\mbM\) is motivated by the fact that we are actually interested in bifurcations that occur in generic non-local (in parameter) families \(V\colon\mcB\to\widehat\Vect\), \(\mcB\subset\bbR^{k}\), and a generic \(k\)-parameter family with \(k<\codim\mbM\) has no common points with~\(\mbM\).

Another assumption required to transfer instability theorems to bifurcations of non-local (in parameter) families is that our submanifold~\(\mbM\) should be \emph{topologically distinguished} in the following sense.
\begin{definition}
    [{cf.~\cite[Definition 16]{IKS-th1}}]%
    \label{def:top-disting}
    We say that a Banach submanifold \(\mbM\subset\widehat\Vect\) is \emph{topologically distinguished} in its neighbourhood~\(\mcU\supset\mbM\), if two vector fields \(v\in\mbM\) and \(w\in\mcU\setminus\mbM\) cannot be orbitally topologically equivalent with \(Sep\)-tracking.
\end{definition}

The following theorem almost immediately follows from the definitions, see also~\cite[Sec. 2.2.2]{GKS-glasses}.
\begin{theorem}
    Consider a Banach submanifold \(\mbM\subset\widehat\Vect\), \(\codim \mbM<\infty\).
    Suppose that \(\mbM\) is topologically distinguished in its neighbourhood~\(\mcU\), and \(\mbM\) has an invariant function \(\varphi\colon\mbM\to \bbR\).

    Then there exists a non-empty open subset~\(\mbM^{\pitchfork}_{nonloc.}\) of the space of \(k\)-parameter non-local families~\(V\colon\mcB\to\widehat\Vect\), \(\mcB\subset\bbR^{k}\), \(k=\codim \mbM\), such that the classification of families from this set has a numerical invariant.
\end{theorem}
\begin{proof}
    First, we define the subset \(\mbM^{\pitchfork}_{nonloc.}\).
    We say that \(V\in\mbM^{\pitchfork}_{nonloc.}\), if
    \begin{itemize}
      \item \(v_{\alpha}\in\mcU\) for all \(\alpha\in\mcB\);
      \item \(V\) meets~\(\mbM\) at a single vector field~\(v\);
      \item \(V\) is transverse to~\(\mbM\) at~\(v\).
    \end{itemize}

    Clearly, \(\mbM^{\pitchfork}_{nonloc.}\) is an open non-empty set.
    Since \(\mbM\) is topologically distinguished, the map \(\mbM^{\pitchfork}_{nonloc.}\to\mbM^{\pitchfork,k}\) given by \(V\mapsto(V, V\cap\mbM)\) sends equivalent non-local families to equivalent local families.
    Therefore, the formula \(V\mapsto\varphi(V\cap\mbM)\) defines a numerical invariant of classification of non-local families~\(V\in\mbM^{\pitchfork}_{nonloc.}\).
\end{proof}

One can apply this theorem to the assertions of Theorems~\ref{thm:pc-n-m} and~\ref{thm:LEG-num} in this paper, and turn them into results on non-local families.
\subsection{Functional invariants}%
\label{sub:func-invs}

In this section we explain why presence of more than one independent robustly invariant function \(\varphi_{j}\colon\mbM\to\bbR\) provides us with functional invariants of classification.
For a more detailed explanation see~\cite[Sec. 3.4]{IKS-th1}.
Informally, in families with sufficiently many parameters, we can express some invariants as functions of other invariants, and these expressions are functional invariants.

\begin{definition}
    We say that invariant functions \(\varphi_{j}\colon\mbM\to\bbR \), \(j=1,\dotsc,D\), are \emph{independent}, if the differential of the map  \(\Phi\colon v\mapsto (\varphi _{1}(v),\dotsc ,\varphi_{D}(v))\) has the full rank at every point \(v\in\mbM\).
\end{definition}

Consider a topologically distinguished submanifold \(\mbM\subset\widehat\Vect\) with $D$ independent robustly invariant functions; put  \(\Phi(v) = (\varphi_1(v),\dotsc, \varphi_D(v))\).
Consider a family with \emph{more} parameters than $\codim \mbM$; namely take \(0<d<D\), \(k:=d+\codim \mbM\), and consider a family \(V\in \mbM^{\pitchfork ,k}\).
Note that it intersects $\mbM$ on a $d$-parameter subfamily.
In the parameter space, this intersection defines the germ
\[
    (S_{V}, 0):=V^{-1}(\mbM)=\set{\alpha \in (\bbR^{k},0)|v_{\alpha }\in \mbM}
\]
of a \(d\)-dimensional submanifold \((S_{V}, 0)\subset (\bbR^k, 0)\).
Consider the germ
\begin{align}
  \label{eq:f_V}
  f_V\colon (S_{V}, 0)&\to (\bbR^{D}, \Phi(v_{0})) &f_V(\alpha )&=\Phi(v_{\alpha})=(\varphi _{1}(v_{\alpha }),\dotsc ,\varphi _{D}(v_{\alpha })).
\end{align}

We claim that this germ considered up to a continuous change of variable in the domain is an invariant of the classification of local families \(V\in\mbM^{\pitchfork, k}\).
Indeed, consider another local family \(\tilde V\in\mbM^{\pitchfork, k}\) equivalent to \(V\).
Since \(\mbM\) is topologically distinguished, we have \(h(S_{V})=S_{\tilde V}\).
Note that for every \(\alpha\in S_{V}\) the germ of \(V\) at \(\alpha\) is equivalent to the germ of \(\tilde V\) at \(h(\alpha)\), hence
\begin{equation}
    \label{eq:f_V-eq}
    f_{V}(\alpha)=\Phi(v_{\alpha})=\Phi(\tilde{v}_{h(\alpha)})=f_{\tilde{V}}(h(\alpha)), \qquad \alpha\in(S_{V}, 0).
\end{equation}
Thus \(f_{\tilde{V}}\) differs from \(f_{V}\) by the continuous change of coordinate in its domain given by \(\chi_{V, \tilde V}:=\left.h\right|_{S_{V}}\colon (S_{V}, 0)\to(S_{\tilde V}, 0)\).
\begin{definition}%
    \label{def:func-inv-rank}
    We say that classification of families \(V\in\mbM^{\pitchfork, k}\) admits \emph{a functional invariant of rank \((d, D)\)}, if
    \begin{itemize}
      \item for every family \(V\in\mbM^{\pitchfork, k}\) there is a germ of a map \(f_{V}\colon (S_{V}, 0)\to (\bbR^{D}, f_{V}(0))\) from a \(d\)-dimensional submanifold \((S_{V}, 0)\subset(\bbR^{k}, 0)\) to \(\bbR^{D}\);
      \item for equivalent families \(V, \tilde V\) there exists a local homeomorphism \(\chi_{V, \tilde V}\colon(S_{V}, 0)\to(S_{\tilde V}, 0)\) such that \(f_{\tilde V}(\chi(\alpha))=f_{V}(\alpha)\) for all \(\alpha\in(S_{V}, 0)\);
      \item there is a non-empty open subset \(U\subset\bbR^{D}\) such that all equivalence classes of germs \(f\colon(S, 0)\to(\bbR^{D}, a)\) with \(a\in U\) can be realized as \(f_{V}\) for some \(V\in\mbM^{\pitchfork, k}\).
    \end{itemize}
\end{definition}

In terms of this definition, we have the following Proposition.
\begin{proposition}
    [cf. {\cite[Proposition 4]{IKS-th1}}]%
    \label{prop:func-invariants}
    Suppose that a topologically distinguished submanifold \(\mbM\subset\widehat\Vect\) has \(D>1\) independent robustly invariant functions \(\varphi _{1},\dotsc ,\varphi_{D}\).
    Take \(0<d<D\), \(k:=d+\codim \mbM\).
    Then the classification of families \(\mbM^{\pitchfork, k}\) admits a functional invariant of rank \((d, D)\).
\end{proposition}
Indeed, \(f_{V}\) given~by~\eqref{eq:f_V} is an invariant of classification due~to~\eqref{eq:f_V-eq}, and the realizability assertion with \(U=\Phi(\mbM)\) easily follows from the independence of \(\varphi_{j}\).

\begin{remark}
    The quotient space\footnote{Elements of this quotient space are called \emph{simple diagrams of rank \((d, D)\)} in~\cite{IKS-th1}.} of the germs \(f_{V}\) by the equivalence relation introduced in~\autoref{def:func-inv-rank} is at least as large as the space of germs of maps \(g\colon(\bbR^{d}, a)\to (\bbR^{D-d}, b)\), \((a, b)\in \Phi(\mbM)\).
    Indeed, for a generic local family~\(V\), the image \(f_{V}(S_{V})\) is a germ of a smooth \(d\)-dimensional submanifold in \(\bbR^{D}\), and this submanifold is the graph of a germ \(g_{V}\colon(\bbR^{d}, a)\to (\bbR^{D-d}, b)\), \((a, b)=\Phi(v_{0})\).
    Clearly, any germ of this type can be realized as \(g_{V}\) for some \(V\).
\end{remark}

We conclude that the existence of at least two independent robust invariants implies existence of functional invariants.

\subsection{Ears and glasses}%
\label{sub:ears-and-glasses}
Consider a~vector field~\(v\in \Vect\).
Recall that the \emph{characteristic number} of a hyperbolic saddle~\(L\) of a vector field~\(v\) is the absolute value of the ratio of the eigenvalues of \(dv(L)\), the negative one is in the numerator.

A~\emph{right lense} of~\(v\) is a tuple of
\begin{itemize}
  \item a~hyperbolic saddle point~\(R\) with the characteristic number \(\rho < 1\);
  \item a~separatrix loop~\(r\) of~\(R\);
  \item a~hyperbolic saddle point~\(I\) and its separatrix~\(\gamma \) that winds onto~\(r\) in the reverse time: \(\alpha (\gamma )=r\cup \set{R}\);
  \item a hyperbolic attractor (either a sink, or an attracting cycle) that attracts the unstable separatrix of~\(R\) that is not a part of~\(r\).
\end{itemize}
The yet unused stable separatrix~\(s\) of~\(R\) is called the \emph{incoming separatrix} of this right lense.
Denote by \(\mcR=(R, r, I, \gamma , s)\) the right lense described above.
The notation omits the hyperbolic attractor that attracts a separatrix of~\(R\), because we will never use it.

The loop~\(r\) splits the sphere into two discs; the one that includes the other two separatrices of~\(R\) is called the \emph{exterior} of~\(r\), and the other one is called the \emph{interior} of~\(r\).
It is easy to see that \(I\) and~\(\gamma \) are located inside~\(r\).

A \emph{left lense} \(\mcL=(L, l, I, \gamma , u)\) of~\(v\) is a right lense of~\(-v\).
Denote by \(\lambda >1\) the characteristic number of~\(L\).

Consider a vector field \(v\) with a left lense~\(\mcL=(L, l, I_{L}, \gamma_{L}, u)\) and a right lense~\(\mcR=(R, r, I_{R}, \gamma_{R}, s)\).
If \(u\) forms a separatrix connection~\(b\) (from \enquote{bridge}) with~\(s\), then we say that \(v\) has a separatrix graph \enquote{ears} or \enquote{glasses} depending on the orientations of \(l\) and \(r\), see \autoref{fig:ears-glasses}.
Let \(\mbE\sqcup\mbG\subset\widehat\Vect\) be the set of vector fields with \enquote{ears} or \enquote{glasses} with marked saddles \(L\), \(R\), \(I_{L}\), and \(I_{R}\), and their separatrices.
Define \(\varphi\colon\mbE\sqcup\mbG\to\bbR\) by
\begin{equation}
    \label{eq:eg:phi}
    \varphi(v)=-\frac{\ln\lambda}{\ln\rho}.
\end{equation}

We will use the following theorem.
\begin{theorem}
    [{\cite[Theorem 5]{GKS-glasses}}]%
    \label{thm:ears-glasses}
    The function \(\varphi\) given by~\eqref{eq:eg:phi} is a robustly invariant function on \(\mbE\sqcup\mbG\).
\end{theorem}
The proof in~\cite{GKS-glasses} relies on the fact that on a certain distinguished one-parameter subfamily in the unfolding, the order in which separatrix connections occur is governed by $-\ln  \lambda/\ln \rho$.
\subsection{Invariants and limit points}
Most examples in this paper follow the same pattern.
We take a degenerate vector field~\(v_{0}\) such that while \(v_{0}\) itself has no \enquote{ears} or \enquote{glasses}, they appear in arbitrarily small perturbations of \(v_{0}\).
Then we claim that the ratio of logarithms of characteristic numbers of the saddles that will participate in these \enquote{ears} or \enquote{glasses} is a robustly invariant function.
\begin{theorem}%
    \label{thm:eg-closure}
    Let \(\mbM\subset\widehat\Vect\) be a Banach submanifold of finite codimension.
    Suppose that for every \(V=\set{v_{\alpha}}\in\mbM^{\pitchfork, k}\), \(k\ge\codim\mbM\), we have four families of marked saddles \(L(\alpha)\), \(R(\alpha)\), \(I_{L}(\alpha)\), \(I_{R}(\alpha)\) such that for some sequence of parameter values \(\alpha_{n}\to 0\) the vector field \(v_{\alpha_{n}}\) has \enquote{ears} or \enquote{glasses} with \(L(\alpha_{n})\), \(R(\alpha_{n})\), \(I_{L}(\alpha_{n})\), \(I_{R}(\alpha_{n})\) playing the same roles as in \autoref{thm:ears-glasses}.
    Then \(\varphi\colon\mbM\to\bbR\), \(\varphi(v)=\frac{\ln\lambda(v)}{\ln\rho(v)}\), where \(\lambda\) and \(\rho\) are characteristic numbers of \(L\) and \(R\), is a robustly invariant function.
\end{theorem}
Recall that the inequality $\lambda(v)>1>\rho(v)$ is required in the definitions of \enquote{glasses} and \enquote{ears}.
\begin{proof}
    Consider two equivalent families \(V, \tilde V\in\mbM^{\pitchfork, k}\).
    Let \((\alpha_{n})\) be the sequence in the parameter space of~\(V\) guaranteed by the assumptions of the theorem.
    For simplicity, put \(L_{n}=L(\alpha_{n})\) etc.
    Let \(H_{n}\colon S^{2}\to S^{2}\) be the homeomorphism of the sphere that implements the equivalence between \(v_{\alpha_{n}}\) and \(\tilde v_{h(\alpha_{n})}\).

    Since \(v_{\alpha_{n}}\) is equivalent to~\(\tilde v_{h(\alpha_{n})}\), the latter vector field has a separatrix graph \enquote{ears} or \enquote{glasses} with \(\tilde L_{n}=H_{n}(L_{n}), \tilde R_{n}=H_{n}(R_{n}), H_{n}(I_{L,n}), H_{n}(I_{R,n})\) playing the role of \(L\), \(R\), \(I_{L}\), \(I_{R}\) in \autoref{sub:ears-and-glasses}.
    Since \(V\) is equivalent to \(\tilde V\), the unfoldings \((V, \alpha_{n})\) and \((\tilde V, \tilde\alpha_{n})\) are equivalent as well.
    Now \autoref{thm:ears-glasses} implies that \(\frac{\ln\lambda_{n}}{\ln\rho_{n}}=\frac{\ln\tilde \lambda_{n}}{\ln\tilde \rho_{n}}\), where \(\lambda_{n}\), \(\rho_{n}\), \(\tilde \lambda_{n}\), \(\tilde\rho_{n}\) are the characteristic numbers of \(L_{n}\), \(R_{n}\), \(\tilde L_{n}\), and \(\tilde R_{n}\), respectively.
    Finally, taking the limit as \(n\to\infty\) completes the proof.
\end{proof}

\section{Adding a semi-stable cycle}%
\label{sec:semi-stable}

\subsection{Cheap numerical invariants}%
\label{sub:pc:many-lenses}
Consider a vector field \(v_0\in \Vect\).
Suppose that it has
\begin{itemize}
  \item a semi-stable cycle \(c\) of multiplicity~\(2\);
    we choose \(\infty \in S^2\) so that \(c\) attracts from exterior and repells from interior;
  \item an \(M\)-tuple~\(\mcL=(\mcL_1, \dotsc , \mcL_{M})\), \(M\ge 1\), of left glass lenses \(\mcL_{i}=(L_{i}, l_{i}, I_{L,i}, \gamma _{L,i}, u_{i})\) located outside~\(c\) such that \(u_{i}\) wind onto~\(c\): \(\omega (u_{i})=c\);
  \item an \(N\)-tuple~\(\mcR=(\mcR_1, \dotsc , \mcR_{N})\), \(N\ge 1\), of right glass lenses \(\mcR_{j}=(R_{j}, r_{j}, I_{R,j}, \gamma _{R,j}, s_{j})\) located inside~\(c\) such that \(s_{j}\) wind onto~\(c\) in backward time: \(\alpha (s_{j})=c\).
\end{itemize}

\begin{figure}
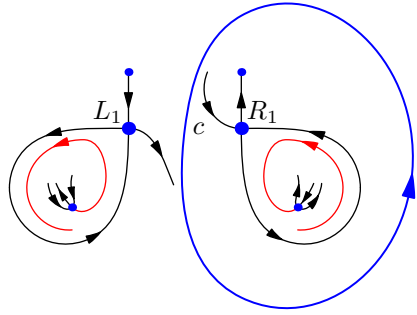

    \centering
    \asyinclude{glasses-whirl}
    \caption{A vector field of class \(\WG_{1,1}\)}\label{fig:whirled-glasses}
\end{figure}

Let \((\WG_{M,N})\subset \widehat{\Vect}\) be the Banach submanifold of vector fields as above with marked saddles $L_i, R_i, I_{L, i}, I_{R,i} $ and their separatrices, see \autoref{fig:whirled-glasses} for~\(M=N=1\).
This submanifold has codimension~\(M+N+1\).

Let $\lambda_j>1$, $\rho_j<1$ be the characteristic numbers of the saddles $L_j, R_j$ respectively.
\begin{theorem}%
    \label{thm:pc-n-m}
    For each $M\ge 1, N\ge 1$, the Banach submanifold $\WG_{M,N}$ of codimension $M+N+1$ has $M+N-1$ independent robustly invariant functions.

    Namely, for each $i\le M$, \(j\le N\), the ratio $\varphi_{i,j}(v)=\frac{\ln \lambda_i(v)}{\ln \rho_j(v)}$ is a robustly invariant function.
\end{theorem}
Clearly, only $M+N-1$ of these ratios are independent; we could also say that the vector function $v\mapsto [\ln \lambda_1(v):\dots:\ln \lambda_M(v): \ln \rho_1(v): \dots: \ln \rho_N(v)]\in\bbR\bbP^{M+N-1}$ is a robustly invariant function.

\begin{proof}
    [Proof of \autoref{thm:pc-n-m}]
    Let us prove that \(\varphi_{i,j}\) is a robustly invariant function.
    Without loss of generality we may and will assume that \(i=j=1\).

    Let as apply \autoref{thm:eg-closure}.
    Namely, each $V\in \WG_{N,M}^{\pitchfork, k}$ has marked saddles $L_1, I_{L, 1}, R_1, I_{R, 1}$ and their local separatrices.
    Put $L_1(\alpha):=L_1(v_\alpha)$ etc.

    Consider a sequence of perturbations $v_{\alpha_k}$ of $v_0$ in $V$ with $\alpha_k\to 0$, such that:
    \begin{itemize}
      \item the glass lenses \(\mcL_1\) and \(\mcR_1\) survive;
      \item the parabolic cycle $c$ disappears;
      \item \(u_1(\alpha)\) makes many turns around the disappeared parabolic cycle and  coalesces with \(s_1(\alpha)\).
    \end{itemize}
    A sequence of this type exists because $V$ is transverse to $\WG_{N,M}$.

    Note that for  $v_{\alpha_k}$, the left lense $ \mcL_1(\alpha)$ and the right lense $\mcR_1(\alpha)$ form either the \enquote{glasses} or the \enquote{ears} separatrix graph.
    Finally, \autoref{thm:eg-closure} implies that $\varphi_{1,1}(v)=\frac{\ln \lambda_1(v)}{\ln \rho_1(v)}$ is a robustly invariant function.
\end{proof}

Due to \autoref{prop:func-invariants}, robustly invariant functions produce functional invariants in families when we add extra parameters.
Thus \autoref{thm:pc-n-m} and \autoref{prop:func-invariants} immediately imply the following corollary.
\begin{corollary}%
    \label{cor:semi-stable-functional}
    For any two natural numbers \(0<d<D\), there exists a submanifold \(\mbM_{d,D}\subset\widehat\Vect\) of codimension \(D+2\) such that the classification of families \(V\in\mbM_{d,D}^{\pitchfork, d+D+2}\) has a functional invariant of rank \((d, D)\).

    Namely, we can choose any \(M, N\ge 1\), \(M+N=D+1\), and take \(\mbM_{d,D}=\WG_{M,N}\).
    In particular, the classification of five-parameter families $V\in\WG_{1,2}^{\pitchfork, 5}$ has a functional invariant of rank \((1, 2)\).
\end{corollary}

Another construction that leads to functional invariants is described in~\cite[Sec. 3]{IKS-th1}.
Namely, it provides \((d+2D+1)\)-parameter families such that their classification up to the moderate equivalence admits functional invariants of rank \((d, D)\).
So for large \(D\) the invariants provided by \autoref{cor:semi-stable-functional} are almost twice \enquote{cheaper} than those constructed in~\cite{IKS-th1}.
In \autoref{sec:lips} we construct numerical and functional invariants in families with even fewer parameters.

\subsection{Infinitely many numerical invariants}%
\label{sub:semi-stable:3-param}
In the settings of the previous section, assume that~\(k=M+N+1\) and consider a family $V\in \WG_{M,N}^{\pitchfork, k}$.

Let \((\mcE, 0)\subset (\bbR^{k}, 0)\) be the curve in the parameter space of $V$ given by the condition \enquote{all the glass lenses survive}.
As in the general case, we will be interested in the points~\(\alpha \in \mcE\) such that the parabolic cycle is destroyed and \(u_1(\alpha)\) coalesces with~\(s_1(\alpha)\), forming the~bridge~\(b(\alpha)\).
In our case \(k=M+N+1\), we have \(\dim \mcE=1\), and these points form a~sequence \(\alpha _{n}\to 0\) enumerated by the number of turns \(n\) of the bridge~\(b\) around the disappeared cycle~\(c\).
For sufficiently large~\(n\), existence of~\(\alpha _{n}\) follows from simple continuity arguments, and uniqueness follows from, e.g.~\cite[Proposition 2]{Afr_Young}.
Put
\[
    \varphi_{n}(V)=\frac{\ln\lambda(v_{\alpha_{n}})}{\ln\rho(v_{\alpha_{n}})}.
\]

We say that the \emph{tail} of a sequence \((x_{n})\) is its equivalence class with respect to the following equivalence relation:
\((x_{n})\sim(\tilde x_{n})\) if for some \(a\in\bbZ\) and sufficiently large \(n\) we have \(\tilde x_{a+n}=x_{n}\).

\begin{theorem}%
    \label{thm:pc:seq}
    In the settings of \autoref{thm:pc-n-m}, consider the space \(\WG^{\pitchfork, k}_{M,N}\), \(k=M+N+1\).
    Then the tail of the sequence $(\varphi_n(V))$ defined above is an invariant of classification of families \(V\in\WG^{\pitchfork, k}_{M,N}\).
\end{theorem}
\begin{proof}
    The proof if straightforward.
    Due to the $Sep$-tracking property, the homeomorphism~\(h\) of the parameter spaces sends the curve~\(\mcE\) of the family~\(V\) to the similarly defined curve~\({\tilde \mcE}\) of~the family~\({\tilde V}\), and each point \(\alpha _{n}\) to one of the points~\({\tilde \alpha }_{m}\).
    Similarly, \(h^{-1}\) sends~\({\tilde \mcE}\) to~\(\mcE\), and each point~\({\tilde \alpha }_{m}\) to one of the~points~\(\alpha _{n}\).

    Since \(h\) preserves the relative order of \(\alpha _{n}\), for some~\(a \in \bbZ\) and sufficiently large \(n\) we have
    \[
        h(\alpha _{n})={\tilde \alpha }_{a+n}.
    \]
    The germ of~\(V\) at~\(\alpha _{n}\) has \enquote{glasses} or \enquote{ears} formed by $\mcL_1(\alpha_{n})$ and $\mcR_1(\alpha_{n})$, and is orbitally topologically equivalent to the germ of~\(\tilde{V}\) at~\({\tilde \alpha }_{a+n}\).
    Thus \autoref{thm:ears-glasses} implies that \(\varphi_{a+n}(\tilde V)=\varphi_{n}(V)\), hence the sequences \((\varphi_{n}(V))\) and \((\varphi_{n}(\tilde V))\) have the same tail.
\end{proof}
This also implies the following.
\begin{corollary}%
    \label{cor:pc:seq}
    For any vector field  $v_0\in \WG_{1,1}^{\pitchfork, 3}$, generic \(3\)-parameter unfoldings of $v_0$ are equivalent if and only if the corresponding sequences \((\varphi_{n}(V))\) have the same tail.

    So the classification of generic \(3\)-parameter unfoldings of $v_0$ has infinitely many numerical invariants.
\end{corollary}

It is easy to see that many sequences can be realized as invariants, i.e.\ can be represented as \(\varphi_{n}(V)\) for some \(V\in\WG_{M,N}^{\pitchfork, M+N+1}\).
E.g., one can perform a local surgery near \(L_{1}\) and \(R_{1}\) to replace \(\lambda(\alpha)\) and \(\rho(\alpha)\) with any two smooth functions that satisfy $\lambda(\alpha)<1<\rho(\alpha)$, without changing the sequence $\alpha_n$.

However, these invariants are not provided by robust invariant functions, and we do not know if the classification of non-local \(4\)-parameter families \(V\) transverse to \(\WG_{1,1}\) has any invariants at all.

\begin{remark}
    One can define similar sequences \(\varphi _{i,j,n}\) based on another pair of lenses \((\mcL_{i}, \mcR_{j})\) instead of \((\mcL_1, \mcR_1)\).
    In the generic case, there is a well-defined cyclic order on these sequences:
    between \(\varphi _{1,1,n}\) and \(\varphi _{1,1,n+1}\) there is exactly one term of each of the other sequences \(\varphi _{i,j,n}\), and \(\varphi _{i,j,n}\) appear in the same order for each \(n\).
    Therefore, if we enumerate all these sequences so that \(\varphi _{i,j,n}\) is between \(\varphi _{1,1,n}\) and \(\varphi _{1,1,n+1}\), and similarly for~\({\tilde V}\), then the number~\(a\) provided by \autoref{thm:pc:seq} is the same for all the sequences \(\varphi _{i,j,n}\).
\end{remark}

\section{Ensemble “lips” that generates \enquote{glasses}}%
\label{sec:lips}

\subsection{Saddle-nodes}%
\label{sub:saddle-nodes}

Recall that a \emph{saddle-node} of a planar vector field~\(v\) is a zero of~\(v\) such that the linearization of~\(v\) at this point has exactly one zero eigenvalue.
In the eigenvector basis we have \(\dot x=X(x, y)=O(x^{2}+y^{2})\), \(\dot y=Y(x, y)=\lambda y+\dotsb\), \(\lambda\neq0\).
Generically, the second derivative \(\frac{\partial^{2} X}{\partial x^{2}}\) is not zero, and the vector field is linearly conjugate to \(\dot x=x^{2}+\dotsb\), \(\dot y=\lambda y+\dotsb\).

If \(\lambda>0\), then the only trajectory that converges to $(0,0)$ in the future is called the \emph{stable separatrix} of the saddle-node.
The trajectories that converge to $(0,0)$ in the past form a \emph{repelling parabolic sector} of a saddle-node.
Two trajectories on the boundaries of the parabolic sector are called \emph{unstable separatrices} of a saddle-node.
Similarly, a saddle-node with \(\lambda<0\) has one unstable separatrix and two stable separatrices that bound an \emph{attracting parabolic sector} of a saddle-node.

We shall use the following finitely smooth local normal form of vector fields with saddle-nodes.
\begin{theorem}
    [{Ilyashenko, Yakovenko, see~\cite[Theorem 5]{IYa91}}]%
    \label{thm:IYa}
    For any \(r\) there exists \(N\) such that any \(C^{N}\) smooth unfolding of a saddle-node is \(C^{r}\) conjugate to the following normal form:
    \begin{equation}
        \label{eq:SN-unfolding}
        \begin{aligned}
            \dot x &= \varepsilon(\alpha)+x^{2}+a(\alpha)x^{3};\\
            \dot y &= \lambda(x, \alpha)y,
        \end{aligned}
    \end{equation}
    where \(\varepsilon\), \(a\), and \(\lambda\) are smooth functions, \(\lambda(0, 0)=\lambda\neq 0\), \(\varepsilon(0)=0\).
\end{theorem}
This theorem is formulated in~\cite{IYa91} for infinitely smooth families of vector fields.
However, this assumption is used only once to apply a theorem by \citeauthor{Takens_1971}~\cite[Sec.\,4]{Takens_1971}.
The latter theorem is formulated for infinitely smooth vector fields as well but this assumption is used in the proof only once to construct finitely smooth invariant manifolds: center, center-stable, and center-unstable manifolds.
This can be done for finitely smooth vector fields too, see, e.g.,~\cite[Sec. 1.3]{Chow_Li_Wang_1994}.
Thus both theorems \cites[Sec.\,4]{Takens_1971}[Theorem 5]{IYa91} apply to sufficiently smooth vector fields.

For an unfolding normalized as in~\eqref{eq:SN-unfolding}, the phase portrait depends on the signs of \(\varepsilon\) and \(\lambda\).
Namely,
\begin{itemize}
  \item for \(\eps(\alpha)<0\), \(v_{\alpha}\) has a saddle and a node near the origin;
  \item for \(\eps(\alpha)=0\), \(v_{\alpha}\) has a saddle-node;
  \item for \(\eps(\alpha)>0\), \(v_{\alpha}\) has no singular points near the origin.
    The correspondence map from \(\Gamma^{-}=\set{x=-1}\) to \(\Gamma^{+}=\set{x=1}\) has the form \(\Delta(y)=C(\eps)y\).
    In the case \(\lambda(0)>0\), we have \(C(\eps)\to \infty\) as \(\eps(\alpha)\to 0+\).
    In the other case \(\lambda(0)<0\), we have \(C(\eps)\to 0\) as \(\varepsilon(\alpha)\to 0+\).
\end{itemize}

\begin{remark}%
    \label{rem:affine}
    Consider a vector field \(v_{0}\) having a~saddlenode of the form \(\dot x=x^{2}+\dots, \dot y=-y+\dots\).
    The rectifying chart provided by \autoref{thm:IYa} is not uniquely defined, and different unfoldings of the same vector field can produce different rectifying charts for \(\alpha=0\).
    However, all these rectifying charts define the same \emph{affine structure} on the cross-section \(\Gamma^{-}=\set{x=-1}\).
    More precisely, given rectifying charts of two unfoldings of \(v_{0}\), the correspondence map along \(v_{0}\) between \(\Gamma^{-}\) of the first chart to \(\Gamma^{-}\) of the second chart is affine in these charts.

    Indeed, for  two segments \(I, J \subset \Gamma^{-}\), the ratio of their lengths in any rectifying chart can be determined using the dynamics of \(v_{0}\) only, in the following way. Take the saturations of \(I\) and \(J\) by forward trajectories  of \(v_{0}\), and consider the intersections of these saturations with \(x=-\eps\).
    Then for any rectifying chart \(\xi\), the ratio of the lengths of these intersections tends to \(|\xi(I)|:|\xi(J)|\).
    This follows from the fact that this limit does not change when normalizing the family.

    From now on, for a saddlenode, we will work in the rectifying chart on \(\Gamma^{-}\) (for \(\lambda<0\)) and on \(\Gamma^{+}\) (for \(\lambda>0\));
    this chart is uniquely determined by \(v_{0}\) up to an affine coordinate change as discussed above. 
\end{remark}

\begin{remark}%
    \label{rem:smooth}
    In this section we replace the smoothness assumption in the definitions of \(\Vect\), \(\mcV\) etc.\ with an assumption \enquote{\(v_{\alpha}(x)\) is \(C^{N}\) smooth in \((\alpha, x)\)} for some unspecified large \(N\).
    We assume that for a \(C^{N}\)-smooth unfolding of a saddlenode there exists a \(C^{1}\) coordinate change that brings it to the form~\eqref{eq:SN-unfolding}.
    This property holds for sufficiently large \(N\) due to \autoref{thm:IYa}.
    It seems that actually \(N=3\) still works but we failed to find a reference for this fact.
\end{remark}

\subsection{Ensemble \enquote{lips}}\label{sub:ensemble-lips}

\begin{definition}%
    \label{def:lips}
    A vector field is said to have \emph{ensemble \enquote{lips}}, see \autoref{fig:lips-orig}, if
    \begin{itemize}
      \item it has two saddlenodes \(S_{1}\) and \(S_{2}\);
      \item \(S_{1}\) has a repelling parabolic sector;
      \item \(S_{2}\) has an attracting parabolic sector;
      \item the only stable separatrix of \(S_{1}\) coalesces with the only unstable separatrix of~\(S_{2}\) forming the bridge~\(b\);
      \item a nonempty open set of trajectories visit both parabolic sectors.
    \end{itemize}
\end{definition}
Choose cross-sections \(\Gamma_1\), \(\Gamma_2\) near \(S_1\), \(S_2\) respectively that intersect all trajectories in parabolic sectors of $S_1$ and $S_2$.
The last requirement means that the correspondence map $P\colon \Gamma_1\mapsto \Gamma_2$ is defined on a non-empty open set.
We will always choose orientations on \(\Gamma_1\) and \(\Gamma_2\) so that $P$ preserves orientation.

\begin{figure}
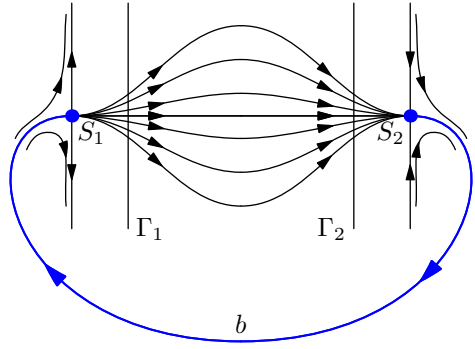

    \centering
    \asyinclude{lips}
    \caption{Ensemble \enquote{lips}}\label{fig:lips-orig}
\end{figure}

Vector fields with an ensemble \enquote{lips} and no other degeneracies form a Banach submanifold of codimension~\(3\) in \(\Vect\).
In the simplest case, the intersection~\(U\) of  attracting and repelling basins of the parabolic sectors of \(R\) and \(L\) contains no other singular points of~\(v\).
Possible bifurcations of vector fields of this type were studied in~\cite{KS}.
In particular, \citeauthor{KS} proved that these bifurcations can generate arbitrarily many limit cycles.
We will use the following lemma from~\cite{KS}.
\begin{lemma}%
    [{cf.~\cite[Sec. 3]{KS}}]%
    \label{lem:affine-approx}
    Let \(v_{0}\) be a vector field with ensemble \enquote{lips}.
    Let \(V=\set{v_{\alpha}}\) be an unfolding of \(v_{0}\) transverse to the codimension \(3\) Banach submanifold of vector fields with ensemble \enquote{lips}.
    Let \(\Gamma_{1}\), \(\Gamma_{2}\) be the cross-sections introduced above with rectifying charts as~in~\autoref{thm:IYa}.
    Then for any affine map \(A\colon\Gamma_{2}\to\Gamma_{1}\) there exists a sequence \(\alpha_{n}\to 0\) such that the correspondence map \(\Delta_{\alpha_{n}}\) tends to \(A\) in \(C^{1}\) norm as \(n\to\infty\).
\end{lemma}
\begin{proof}
    Note that \(\Delta_{\alpha}=\Delta_{L,\alpha}\circ\Delta_{b,\alpha}\circ\Delta_{R,\alpha}\), where \(\Delta_{R,\alpha}\colon\Gamma_{2}\to\Gamma_{2}'\) is the correspondence map through the disappeared saddlenode~\(S_{2}\), \(\Delta_{b,\alpha}\colon\Gamma_{2}'\to\Gamma_{1}'\) is the correspondence map along the bridge~\(b\), and \(\Delta_{L,\alpha}\colon\Gamma_{1}'\to\Gamma_{1}\) is the correspondence map through the disappeared saddlenode \(S_{1}\).
    Recall that \(\Delta_{R,\alpha}(x)=C_{1}(\alpha)x\), \(C_{1}\ll 1\), and \(\Delta_{L,\alpha}(x)= C_{2}(\alpha)x\), \(C_{2}\gg 1\).
    Note that \(\Delta_{b,\alpha}(x)\approx \Delta_{b,\alpha}(0)+\Delta_{b,0}'(0)x\), hence \(\Delta_{\alpha}(x)\approx C_{1}(\alpha)\Delta_{b,0}'(0)C_{2}(\alpha)x+\Delta_{b,\alpha}(0)C_{2}(\alpha)\).
    Due to the transversality condition, we can change \(C_{1}\), \(C_{2}\), and \(\Delta_{b,\alpha}(0)\) independently, thus for some \(\alpha_{n}\to 0\), the map \(\Delta_{\alpha_{n}}(x)\) tends to  any prescribed affine map.
\end{proof}

In this section we study a modification of the ensemble \enquote{lips} that leads to arbitrarily many numerical invariants of classification of \(4\)-parameter infinitely smooth families of vector fields.
Namely, we will take \(n\) Cherry cells and one right lense, and use local surgery to paste them into the phase portrait of the vector field described above.

\begin{definition}
    A \emph{Cherry cell} is a vector field \(v\) in the square \([0,1]\times  [0,1]\) such that:

    \begin{itemize}
      \item \(v\) is equal to \((1, 0)\) in a neighbourhood of the boundary of the square;
      \item \(v\) has one saddle and one repellor (an unstable node or an unstable focus), and no other singularities;
      \item both unstable separatrices of the saddle intersect the right side \(\set{1}\times  [0,1]\) of the square;
      \item one stable separatrix intersects \(\set{0}\times  [0,1]\), the other one enters the repellor in the reverse time.
    \end{itemize}
\end{definition}
We shall say that the stable separatrix that intersects \(\set{0}\times[0,1]\) is the \emph{incoming} separatrix of the Cherry cell.
The first property makes Cherry cells useful for \emph{local surgery}: one can take a flow box of a vector field, and replace it with a vector field diffeomorphic to a Cherry cell without changing the vector field outside of the flow box.

\subsection{Surgery on an ensemble \enquote{lips} and the class $\LEG_n$}
Let us take \(n+1\) Cherry cells and one right lense, and paste them into the \enquote{lips}.
More precisely, we consider degenerate vector fields of the following class, see \autoref{fig:lips-modified}.
\begin{definition}
    The vector field \(v\) belongs to the class \(\LEG_n\) if

    \begin{itemize}
      \item it has an ensemble “lips” with saddlenodes \(S_1\), \(S_2\) as described above;
      \item it has \(n+1\) Cherry cells with saddles \(L_1,L_2, \dotsc, L_{n+1}\);
      \item all the characteristic numbers \(\lambda_{k}\) of \(L_{k}\) are greater than one;
      \item the stable separatrix \(l^{s}_{k}\) of \(L_{k}\) enters the parabolic sector of \(S_1\) in the past.
        Denote its intersection with the cross-section $\Gamma_1$ of the parabolic sector by \(x_k:=l^{s}_{k}\cap \Gamma_{1}\);
      \item the unstable separatrices \(u_{k}\) and \(l^{u}_{k}\) of \(L_{k}\) enter the parabolic sector of \(S_2\) in the future.
        Denote their intersections with the cross-section $\Gamma_2$ of  the parabolic sector by \(y_k:=u_{k}\cap\Gamma_{2}, z_k:=l^{u}_{k}\cap\Gamma_{2}\);
      \item \(v\) has a right lense \(\mcR=(R, r, I_{R}, \gamma_{R} , s)\) such that \(s\) enters the parabolic sector of \(S_{1}\) in the past, \(s\cap\Gamma_{1}=\set{x_{0}}\);
      \item \(x_{0}<x_{1}<\dotsb<x_{n+1}\), \quad \(y_{1}<z_{1}<y_{2}<z_{2}<\dotsb<y_{n+1}<z_{n+1}\).
    \end{itemize}
    We consider the saddles \(L_{1}, \dotsc, L_{n+1}\), \(R\), \(I\), and their separatrices as marked.
\end{definition}
\begin{figure}
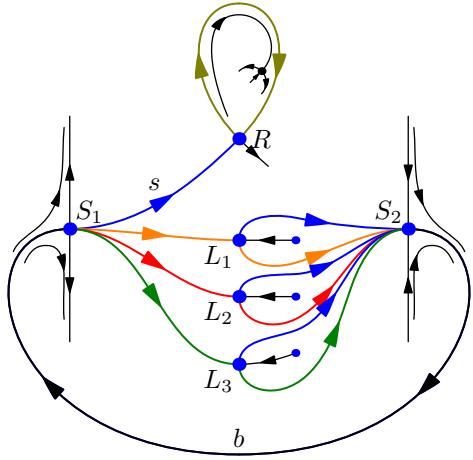

    \centering
    \asyinclude{lips-modified}
    \caption{A vector field of class \(\LEG_{2}\)}\label{fig:lips-modified}
\end{figure}

The last property uses orientations on $\Gamma_1, \Gamma_2$.
Recall that we choose orientation on the cross-sections so that the Poincaré map $P\colon \Gamma_1\to \Gamma_2$ preserves orientation.

\subsection{Numerical and functional invariants for ensemble \enquote{lips}}%
\label{sub:lips-theorems}

The following theorems are formulated for sufficiently smooth families of vector fields, see \autoref{rem:smooth}.

\begin{theorem}%
    \label{thm:LEG-num}
    For each \(n\ge 1\), there exists a non-empty relatively open subset \(\mbM\subset\LEG_n\) that has \(n\) independent robustly invariant functions, namely
    \begin{align*}
      \varphi_{k}(v)&=\frac{\ln\lambda_{k}(v)}{\ln \rho(v)}, & k&=1,\dotsc,n. 
    \end{align*}
\end{theorem}

By \autoref{prop:func-invariants}, this immediately implies the following theorem.
\begin{theorem}%
    \label{thm:LEG-func}
    Given natural numbers \(0<d<n\), the classification of families \(\mbM^{\pitchfork,4+d}\) up to the weak equivalence with \(Sep\)-tracking has functional invariants.
    Namely, to each family \(V\in\mbM^{\pitchfork, 4+d}\) there corresponds a germ of a smooth function \(f_{V}\colon(\bbR^{d}, 0)\to(\bbR^{n}, a)\), and for equivalent families \(V\), \(\tilde V\), the germs \(f_{V}\) and \(f_{\tilde V}\) differ by a continuous change of variable in the domain.
\end{theorem}

\subsection{Technical assumption}%
\label{sub:technical-assumption}
In this section we describe the open set \(\mbM\) from \autoref{thm:LEG-func}.
Take $v\in \LEG_n$.

For each \(k=1,\dotsc,n\), let \(A_{k}\colon \Gamma_{2}\to\Gamma_{1}\) be the unique map such that \(A_{k}(y_{k})=x_{0}\), \(A_{k}(z_{k})=x_{k}\), and \(A_{k}\) is affine in the rectifying charts on \(\Gamma_{2}\), \(\Gamma_{1}\).
Due to \autoref{rem:affine}, these rectifying charts are uniquely defined up to an affine change of coordinates, so the latter assumption depends only on \(v_{0}\), not on (a rectifying chart of) its unfolding \(v_{\alpha}\).

The Poincaré map  \(P\colon \Gamma_{1}\to\Gamma_{2}\) along $v_0$ has jump discontinuities at the points \(x_{k}\), and we define \(P(x_{k}):=z_{k}\).
Then the point \(z_{k}\) is a fixed point of the map \(P\circ A_{k}\), and we have
\begin{align}
  P(x)&\approx z_{k}+C_{k}{(x-x_{k})}^{\lambda_{k}} & \text{as }x&\to x_{k}+0;\nonumber\\
  (P\circ A_{k})(z)&\approx z_{k}+C_{k}'{(z-z_{k})}^{\lambda_{k}} & \text{as }z&\to z_{k}+0.\label{eq:zk-attracts}
\end{align}
Recall that \(\lambda_{k}>1\), hence \(z_{k}\) strongly attracts orbits of \(P\circ A_{k}\) from the right.

We will prove \autoref{thm:LEG-func} for the subset \(\mbM\subset\LEG_{n}\) from the following lemma.
\begin{lemma}%
    \label{lem:technical}
    Let \(\mbM_{n}\subset\LEG_{n}\) be the set of vector fields such that for each \(k=1,\dotsc,n\) the point \(y_{k+1}\) belongs to the basin of attraction of~\(z_{k}\) with respect to \(P\circ A_{k}\):
    \begin{equation}
        \label{eq:z-attracts-y}
        \lim_{l\to\infty}{(P\circ A_{k})}^{l}(y_{k+1})=z_{k}.
    \end{equation}
    Then for each \(n=0,1,\dotsc\), the set \(\mbM_{n}\) is nonempty and relatively open in \(\LEG_{n}\).
\end{lemma}

\begin{proof}
    Clearly, each set \(\mbM_{n}\subset\LEG_{n}\) is relatively open in \(\LEG_{n}\).
    Let us prove that all these sets are nonempty.
    Briefly speaking, we use local surgery to  add the Cherry cells one by one, and place every \(y_{k+1}\) sufficiently close to~\(z_{k}\).

    First, we construct a vector field \(v\in\LEG_{0}\).
    Take a vector field with ensemble \enquote{lips} and no other degeneracies.
    Take two flow boxes inside the \enquote{lips};
    we paste a right lense and a saddle into one of them, and a Cherry cell into another one.
    Clearly, we can choose orientations on \(\Gamma_{1,2}\) so that this vector field belongs to~\(\LEG_{0}\).
    Note that \(\mbM_{0}=\LEG_{0}\) thus \(\mbM_{0}\) is nonempty.

    Once we have a vector field of class \(\mbM_{0}\), we add other Cherry cells one by one.
    Consider a vector field \(v_{0}\in\mbM_{n-1}\).
    It has \(n\) Cherry cells inside the \enquote{lips}, and~\eqref{eq:z-attracts-y} holds for all \(k=1,\dotsc,n-1\).
    Let us choose a location for the \((n+1)\)-th Cherry cell so that~\eqref{eq:z-attracts-y} holds for \(k=n\).
    Since \(x_{0}, x_{n}, y_{n}, z_{n}\) are already fixed, the affine map \(A_{n}\) is already fixed as well.

    Take a point \(y_{n+1}\in\Gamma_{2}\), \(y_{n+1}>z_{n}\) in the basin of attraction of \(z_{n}\) with respect to \(P_{n}\circ A_{n}\), where \(P_{n}\colon \Gamma_{1}\to\Gamma_{2}\) is the correspondence map for the vector field with \(n\) Cherry cells.
    This is possible due~to~\eqref{eq:zk-attracts}.
    Let \(U\) be the saturation of the interval \((x_{n}, A_{n}(y_{n+1}))\) by the trajectories of~\(v\).
    Since \(y_{n+1}\) is greater than all the points in \(U\cap\Gamma_{2}\), we can choose a flow box inside the lips and paste a Cherry cell into this flow box so that
    \begin{itemize}
      \item the incoming separatrix of the new Cherry cell meets \(\Gamma_{1}\) at a point \(x_{n+1}\), \(x_{n+1}>A_{n}(y_{n+1})\);
      \item the unstable separatrices of the new saddle \(L_{n+1}\) meet \(\Gamma_{2}\) at \(y_{n+1}\) and another point \(z_{n+1}>y_{n+1}\);
      \item the local surgery does not affect the vector field at the saturation of the set \(\set{x\le A_{n}(y_{n+1})}\) by trajectories.
    \end{itemize}
    The new vector field has \((n+1)\) Cherry cells, so it belongs to \(\LEG_{n}\).
    It satisfies \autoref{eq:z-attracts-y} for all \(k=1,\dotsc,n\), thus it belongs to \(\mbM_{n}\).
    By induction, all \(\mbM_{n}\) are nonempty.
\end{proof}
\subsection{Proof of \autoref{thm:LEG-func}}\label{sub:proof}
Consider a vector field \(v_{0}\in\mbM_{n}\), where \(\mbM_{n}\subset\LEG_{n}\) is the same as in \autoref{lem:technical}.
Let \(V=\set{v_{\alpha}}_{\alpha\in(\bbR^{D},0)}\), \(D\ge 4\), be an unfolding of~\(v_{0}\) transverse to \((\LEG_n, v_{0})\).
Fix \(k\), \(1\le k \le n\).
Note that \autoref{thm:LEG-num} immediately follows from \autoref{thm:eg-closure} and the following lemma.
\begin{lemma}%
    \label{lem:LEG-perturb}
    For each \(k=1,\dotsc,n\) there exists an arbitrarily small \(\alpha\) such that for the vector field \(v_{\alpha}\)
    \begin{enumerate}
      \item\label{it:rl-survives} the right lense \(\mcR\) survives;
      \item\label{it:sn-disappear} both saddlenodes disappear;
      \item\label{it:uk-s} the unstable separatrix \(u_{k}(\alpha)\) of \(L_{k}(\alpha)\) coalesces with the stable separatrix \(s(\alpha)\) of \(R(\alpha)\);
      \item\label{it:luk-lsk} the unstable separatrix \(l^{u}_{k}(\alpha)\) of \(L_{k}(\alpha)\) coalesces with the stable separatrix \(l^{s}_{k}(\alpha)\) of the same saddle forming the separatrix loop~\(l_{k}(\alpha)\);
      \item\label{it:ukp1-winds} \(u_{k+1}(\alpha)\) winds onto \(l_{k}(\alpha)\) in the forward time.
    \end{enumerate}
\end{lemma}
Indeed, consider a vector field \(v_{\alpha}\) satisfying all conditions~\ref{it:rl-survives}--\ref{it:ukp1-winds} listed above, see \autoref{fig:lips-unfolded}.
Note that \((L_{k}(\alpha), l(\alpha), L_{k+1}(\alpha), u_{k+1}(\alpha), u_{k}(\alpha))\) is a left lense.
Since the unstable separatrix \(u_{k}(\alpha)\) of this left lense coalesces with the stable separatrix \(s(\alpha)\) of \(\mcR\), these two lenses form \enquote{ears} or \enquote{glasses} depending on their orientation.
Reference to~\autoref{thm:eg-closure} completes the proof.

\begin{figure}
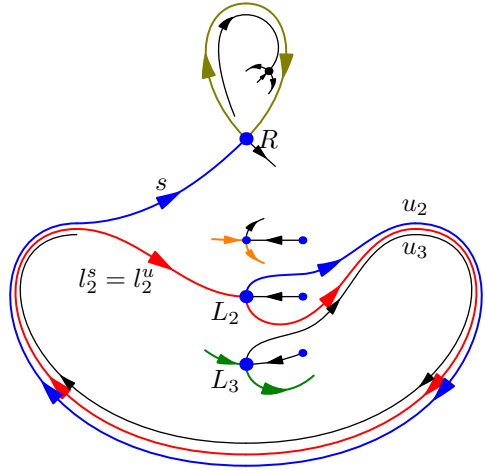

    \centering
    \asyinclude{lips-unfolded}
    \caption{An unfolding of a vector field \(v_{0}\in\LEG_{2}\) satisfying assertions of \autoref{lem:LEG-perturb} for \(k=2\)}\label{fig:lips-unfolded}
\end{figure}

\begin{proof}[Proof of~\autoref{lem:LEG-perturb}]
    The first condition defines a codimension one submanifold in the parameter space.
    From now on, we consider the restriction of our family to this submanifold.
    Despite ambiguity, we continue to use notation \(v_{\alpha}\) for this restriction.

    The second condition can be written as \(\eps_{1}(\alpha)>0\), \(\eps_{2}(\alpha)>0\), where \(\eps_{1}\) and \(\eps_{2}\) are the parameters provided by \autoref{thm:IYa} for \(S_{1}\) and \(S_{2}\), respectively.

    Conditions~\ref{it:uk-s} and~\ref{it:luk-lsk} can be written as \(\Delta_{\alpha}(y_{k}(\alpha))=x_{0}(\alpha)\) and \(\Delta_{\alpha}(z_{k}(\alpha))=x_{k}(\alpha)\), respectively, where \(\Delta_{\alpha}\) is the correspondence map through disappeared saddlenodes and along the bridge~\(b\).
    Due to~\autoref{lem:affine-approx}, \(\Delta_{\alpha}\) can approximate any affine map sending \(y_{k}(0)\) to a point near \(x_{0}(0)\) and \(z_{k}(0)\) to \(x_{k}(0)\).
    Now standard continuity arguments imply that there exists  arbitrarily small \(\alpha\) such that  \(\Delta_{\alpha}(y_{k}(\alpha))=x_{0}(\alpha)\) and \(\Delta_{\alpha}(z_{k}(\alpha))=x_{k}(\alpha)\), so conditions~\ref{it:uk-s} and~\ref{it:luk-lsk} are satisfied.
    Note that the maps \(\Delta_{\alpha}\) tend to the affine map \(A_{k}\) from \autoref{lem:technical} as \(\alpha\to 0\).

    The last condition means that \(y_{k+1}(\alpha)\) belongs to the basin of attraction of \(z_{k+1}(\alpha)\) with respect to the Poincaré map \(P_{\alpha}\circ\Delta_{\alpha}\), where \(P_{\alpha}\colon\Gamma_{1}\to\Gamma_{2}\) is the correspondence map through the lips, \(P_{0}=P\).
    As mentioned above, \(\Delta_{\alpha}\) is close to the affine map \(A_{k}\) from~\autoref{lem:technical}, hence~\eqref{eq:z-attracts-y} guarantees that the last condition is satisfied automatically.
\end{proof}

\section{Future plans}

\subsection{Families with few parameters}%
\label{sub:few-param}
On the one hand, a generic one-parameter family of vector fields on the sphere is structurally stable, see~\cite{MP,St18,YuINS,GIS-semistable}.
On the other hand, classification of locally generic three-parameter families can have a numerical invariant arising from a robustly invariant function, see~\cite{IKS-th1,GKS-glasses}, or a non-robust invariant with values in a larger space, see~\cite{GK-phase} and \autoref{thm:pc:seq}.
The following question is open:
\begin{itemize}
  \item Is a generic 2-parameter family of planar vector fields structurally stable?
\end{itemize}

To the best of our knowledge, all known types of locally generic 2-parameter bifurcations~\cite{Kuznetsov_2004,Roitenberg-thesis:transl,roit12:trans,Shashkov92, Dukov18-heart:en} are structurally stable.
We expect that the answer to the previous question is \enquote{yes}, but the abundant combinatorics of 2-parameter families makes it difficult to classify them.

For $3$- and $4$-parameter families, the following questions remain widely open:
\begin{itemize}
  \item Find more structurally unstable 3-parameter families, and new mechanisms for structural instability.
    All known examples are due to one and the same effect: orders of the parameter values corresponding to sparkling separatrix connections.

  \item Try to find functional invariants for 3-,4- parameter families;

  \item Try to find more mechanisms for the existence of functional invariants.
    All known examples exploit the same idea, see \autoref{prop:func-invariants}.
\end{itemize}

See also~\cite{YuI-surv} for a more detailed survey of current progress on these and other related questions.

\subsection{Local, non-local, and global  bifurcations}%
\label{sub:local-bifurcations}
Bifurcations of singular points of vector fields in their neighborhoods are called \emph{local} bifurcations.
Bifurcations in neighborhoods of polycycles are called \emph{non-local} bifurcations.
Here \enquote{non-local} refers to the space variable, not parameters of a family.
The bifurcations of \enquote{tears of the heart}, \enquote{ears}, \enquote{glasses}, and the bifurcations in this paper are neither local nor non-local:
separatrices that wind onto saddle loops of separatrix graphs play an important role in the bifurcations.
We call them \emph{global} bifurcations.

A natural question is, whether similar effects (structural instability, numerical and functional invariants) occur in local and semi-local bifurcations.

For non-local bifurcations, the answer is \enquote{yes}, see~\cite{DuYuI-nonlocal,Dukov-nonlocal}, while the following questions are widely open.
\begin{itemize}
  \item Can a finite-parameter \emph{local} bifurcation in a generic family be structurally unstable? Can it have functional invariants?
\end{itemize}

\subsection{Vector fields with no versal deformations}%
\label{sub:no-versal}

V.\,Arnold conjectured that any finite-codimension degenerate vector field on \(S^2\) has a \emph{versal deformation};
informally, this is the family in which the bifurcation occurs \enquote{in all possible ways}.

Formally, we say that a local family \(V=\set{v_{\alpha}}_{\alpha\in(\bbR^{k},0)}\) is a \emph{versal deformation} of \(v_{0}\), if any other local family \(\set{\tilde{v}_{\tilde{\alpha}}}_{\tilde{\alpha}\in(\bbR^{\tilde{k}},0)}\) such that \(\tilde{v}_{0}=v_{0}\) is weakly topologically equivalent to a family induced from~\(V\).
Equivalently, \(V\) is a versal deformation, if for any other local family \(\set{\tilde{v}_{\tilde{\alpha}}}_{\tilde{\alpha}\in(\bbR^{\tilde{k}},0)}\) there exists a germ of a continuous map \(h\colon(\bbR^{\tilde{k}}, 0)\to(\bbR^{k},0)\) such that \(v_{h(\tilde{\alpha})}\) is orbitally topologically equivalent to \(\tilde{v}_{\tilde{\alpha}}\) for all \(\alpha\).

The second author has an idea how to construct a counter-example to this conjecture.

\subsection{Do we really need a homeomorphism?}

While theorems in~\cite{IKS-th1,GKS-glasses} require \(h\) to be a homeomorphism, we expect that a similar statement is true for any continuous function~\(h\).
\begin{conjecture}%
    \label{conj:induced}
    Given two local families \(V\in{(\mbE\sqcup\mbG)}^{\pitchfork,k}\) and \(\tilde{V}\in{(\mbE\sqcup\mbG)}^{\pitchfork,\tilde{k}}\), suppose that there exists a~germ of a~continuous map~\(h\colon(\bbR^{k}, 0)\to(\bbR^{\tilde{k}},0)\) such that \(\tilde{v}_{h(\alpha)}\) is equivalent to \(v_{\alpha}\) for all \(\alpha\in(\bbR^{k},0)\).
    Then \(\varphi(v_{0})=\varphi(\tilde{v}_{0})\), where \(\varphi\) is~given~by~\eqref{eq:eg:phi}.
\end{conjecture}

The original proof of \autoref{thm:ears-glasses} relies on computing the \emph{relative density} of two sequences of points in the parameter space.
These proofs fail, if \(h\) is not required to be bijective.
We think that a more detailed study of the bifurcation scenario will allow us to prove \autoref{conj:induced}.

One of the reasons to study this question is that Arnold's definition of a versal deformation involves families related as described in \autoref{conj:induced}.

Another reason to ask the same question comes from the category theory interpretation of invariant functions and functional invariants.
Consider \(\mcV_{k}\) as a category: a morphism between two families \(V\), \(\tilde V\) is a germ of a continuous map \(h\) such that \(\tilde v_{h(\alpha)}\) is equivalent~to~\(v_{\alpha}\).
Two families are equivalent in the sense of \autoref{def:tracking}, if they are isomorphic in this category.

Let \(\SD_{k,d,D}\) be the category of germs of smooth maps \(f\colon(S_{f}, 0)\to(\bbR^{D}, f(0))\) defined on germs of smooth \(d\)-dimensional submanifolds \((S_{f}, 0)\subset(\bbR^{k},0)\).
A morphism between two germs \(f\), \(\tilde f\) is a germ of a continuous map \(\chi\colon(S_{f}, 0)\to(S_{\tilde f}, 0)\) such that \(f=\tilde{f}\circ\chi\).

Let \(\mbM\) be one of the Banach manifolds with robustly invariant functions; e.g., \(\mbM=\mbE\sqcup\mbG\).
The space \(\mbM^{\pitchfork, k}\) inherits the category structure from \(\mcV_{k}\).

Recall that the \emph{core} of a category is the subcategory with the same objects, and morphisms defined to be the isomorphisms of the original category.
The relation~\eqref{eq:f_V-eq} implies that \(V\mapsto f_{V}\) is a functor from the core of \(\mbM^{\pitchfork, k}\) to the core of \(\SD_{k,d,D}\).
In particular, this map sends isomorphic objects to isomorphic objects, i.e., equivalent families of vector fields to equivalent germs.
So it seems natural to ask whether this map actually defines a functor between the original categories as well.
Due to the arguments from \autoref{sub:func-invs}, this question is equivalent to \autoref{conj:induced}.

\section{Acknowledgements}
We are thankful to Yulij Ilyashenko who formulated the problem and provided a continuous support to our work.
We are also grateful to Vladimir Roitenberg, Lev Lerman, Victor Kleptsyn, and Konstantin Khanin for consultations.
Many thanks to Nikita Solodovnikov who came up with \enquote{ears} and \enquote{glasses} separatrix graphs, which is the topic of our previous paper with him and the basis for the current paper.

\printbibliography
\end{document}